\documentclass[12pt]{amsart}
\usepackage{}
\usepackage{amsmath}
\usepackage{amsfonts}
\usepackage{amssymb}
\usepackage[all,cmtip]{xy}           

\usepackage{bbm}
\usepackage{bbding}
\usepackage{txfonts}
\usepackage[shortlabels]{enumitem}
\usepackage{ifpdf}
\ifpdf
\usepackage[colorlinks,final,backref=page,hyperindex]{hyperref}
\else
\usepackage[colorlinks,final,backref=page,hyperindex,hypertex]{hyperref}
\fi
\usepackage{amscd}
\usepackage{tikz-cd}
\usepackage[active]{srcltx}

\makeatletter

\newtheorem{theorem}{Theorem}[section]
\newtheorem{prop}[theorem]{Proposition}
\newtheorem{lemma}[theorem]{Lemma}
\newtheorem{coro}[theorem]{Corollary}
\newtheorem{prop-def}{Proposition-Definition}[section]

\theoremstyle{definition}
\newtheorem{defn}[theorem]{Definition}

\newtheorem{remark}[theorem]{Remark}
\newtheorem{exam}[theorem]{Example}

\newcommand{\nc}{\newcommand}

\nc{\delete}[1]{{}}
\nc{\mmargin}[1]{}

\nc{\mlabel}[1]{\label{#1}}  
\nc{\mcite}[1]{\cite{#1}}  
\nc{\mref}[1]{\ref{#1}}  
\nc{\mbibitem}[1]{\bibitem{#1}} 

\delete{
	\nc{\mlabel}[1]{\label{#1}  
		{\hfill \hspace{1cm}{\bf{{\ }\hfill(#1)}}}}
	\nc{\mcite}[1]{\cite{#1}{{\bf{{\ }(#1)}}}}  
	\nc{\mref}[1]{\ref{#1}{{\bf{{\ }(#1)}}}}  
	\nc{\mbibitem}[1]{\bibitem[\bf #1]{#1}} 
}

 \font\cyrs=wncyr7

\newcommand{\bk}{{\mathbf{k}}}

\nc{\vep}{\varepsilon}
\nc{\bin}[2]{ (_{\stackrel{\scs{#1}}{\scs{#2}}})}  
\nc{\binc}[2]{(\!\! \begin{array}{c} \scs{#1}\\
		\scs{#2} \end{array}\!\!)}  
\nc{\bincc}[2]{  ( {\scs{#1} \atop
		\vspace{-1cm}\scs{#2}} )}  
\nc{\oline}[1]{\overline{#1}}
\nc{\mapm}[1]{\lfloor\!|{#1}|\!\rfloor}
\nc{\bs}{\bar{S}}
\nc{\la}{\longrightarrow}
\nc{\ot}{\otimes}
\nc{\rar}{\rightarrow}
\nc{\lon }{\,\rightarrow\,}
\nc{\dar}{\downarrow}
\nc{\dap}[1]{\downarrow \rlap{$\scriptstyle{#1}$}}
\nc{\defeq}{\stackrel{\rm def}{=}}
\nc{\dis}[1]{\displaystyle{#1}}
\nc{\dotcup}{\ \displaystyle{\bigcup^\bullet}\ }
\nc{\hcm}{\ \hat{,}\ }
\nc{\hts}{\hat{\otimes}}
\nc{\hcirc}{\hat{\circ}}
\nc{\lleft}{[}
\nc{\lright}{]}
\nc{\curlyl}{\left \{ \begin{array}{c} {} \\ {} \end{array}
	\right .  \!\!\!\!\!\!\!}
\nc{\curlyr}{ \!\!\!\!\!\!\!
	\left . \begin{array}{c} {} \\ {} \end{array}
	\right \} }
\nc{\longmid}{\left | \begin{array}{c} {} \\ {} \end{array}
	\right . \!\!\!\!\!\!\!}
\nc{\ora}[1]{\stackrel{#1}{\rar}}
\nc{\ola}[1]{\stackrel{#1}{\la}}
\nc{\scs}[1]{\scriptstyle{#1}} \nc{\mrm}[1]{{\rm #1}}
\nc{\dirlim}{\displaystyle{\lim_{\longrightarrow}}\,}
\nc{\invlim}{\displaystyle{\lim_{\longleftarrow}}\,}
\nc{\dislim}[1]{\displaystyle{\lim_{#1}}} \nc{\colim}{\mrm{colim}}
\nc{\mvp}{\vspace{0.3cm}} \nc{\tk}{^{(k)}} \nc{\tp}{^\prime}
\nc{\ttp}{^{\prime\prime}} \nc{\svp}{\vspace{2cm}}
\nc{\vp}{\vspace{8cm}}
\nc{\modg}[1]{\!<\!\!{#1}\!\!>}
\nc{\intg}[1]{F_C(#1)}
\nc{\lmodg}{\!<\!\!}
\nc{\rmodg}{\!\!>\!}
\nc{\cpi}{\widehat{\Pi}}
\nc{\ssha}{{\mbox{\cyrs X}}} 
\nc{\tsha}{{\mbox{\cyrt X}}}
\nc{\shpr}{\diamond}    
\nc{\labs}{\mid\!}
\nc{\rabs}{\!\mid}

\nc{\ad}{\mrm{ad}}
\nc{\ann}{\mrm{ann}}
\nc{\Aut}{\mrm{Aut}}
\nc{\bim}{\mbox{-}\mathsf{Bimod}}
\nc{\br}{\mrm{bre}}
\nc{\can}{\mrm{can}}
\nc{\Cont}{\mrm{Cont}}
\nc{\rchar}{\mrm{char}}
\nc{\cok}{\mrm{coker}}
\nc{\de}{\mrm{dep}}
\nc{\dtf}{{R-{\rm tf}}}
\nc{\dtor}{{R-{\rm tor}}}

\nc{\Div}{{\mrm Div}}
\nc{\Diff}{\mrm{DA}}
\nc{\Diffl}{\mathsf{DA}_\lambda}
\nc{\diffo}{{\mathsf{DO}_\lambda}}
\nc{\alg}{\mathsf{Alg}}
\nc{\End}{\mrm{End}}
\nc{\Ext}{\mrm{Ext}}
\nc{\Fil}{\mrm{Fil}}
\nc{\Fr}{\mrm{Fr}}
\nc{\Frob}{\mrm{Frob}}
\nc{\Gal}{\mrm{Gal}}
\nc{\GL}{\mrm{GL}}
\nc{\Hom}{\mrm{Hom}}
\nc{\Hoch}{\mrm{Hoch}}
\nc{\hsr}{\mrm{H}}
\nc{\hpol}{\mrm{HP}}
\nc{\id}{\mrm{id}}
\nc{\im}{\mrm{im}}
\nc{\Id}{\mrm{Id}}
\nc{\ID}{\mrm{ID}}
\nc{\Irr}{\mrm{Irr}}
\nc{\incl}{\mrm{incl}}
\nc{\length}{\mrm{length}}
\nc{\NLSW}{\mrm{NLSW}}
\nc{\Lie}{\mrm{Lie}}
\nc{\mchar}{\rm char}
\nc{\mpart}{\mrm{part}}
\nc{\ql}{{\QQ_\ell}}
\nc{\qp}{{\QQ_p}}
\nc{\rank}{\mrm{rank}}
\nc{\rcot}{\mrm{cot}}
\nc{\rdef}{\mrm{def}}
\nc{\rdiv}{{\rm div}}
\nc{\rtf}{{\rm tf}}
\nc{\rtor}{{\rm tor}}
\nc{\res}{\mrm{res}}
\nc{\SL}{\mrm{SL}}
\nc{\Spec}{\mrm{Spec}}
\nc{\tor}{\mrm{tor}}
\nc{\Tr}{\mrm{Tr}}
\nc{\tr}{\mrm{tr}}
\nc{\wt}{\mrm{wt}}

\nc{\bfk}{{\bf k}}
\nc{\bfone}{{\bf 1}}
\nc{\bfzero}{{\bf 0}}
\nc{\detail}{\marginpar{\bf More detail}
	\noindent{\bf Need more detail!}
	\svp}
\nc{\gap}{\marginpar{\bf Incomplete}\noindent{\bf Incomplete!!}
	\svp}
\nc{\FMod}{\mathbf{FMod}}
\nc{\Int}{\mathbf{Int}}
\nc{\Mon}{\mathbf{Mon}}
\nc{\remarks}{\noindent{\bf Remarks: }}
\nc{\Rep}{\mathbf{Rep}}
\nc{\Rings}{\mathbf{Rings}}
\nc{\Sets}{\mathbf{Sets}}
\nc{\ob}{\mathsf{Ob}}

\nc{\BA}{{\mathbb A}}   \nc{\CC}{{\mathbb C}}
\nc{\DD}{{\mathbb D}}   \nc{\EE}{{\mathbb E}}
\nc{\FF}{{\mathbb F}}   \nc{\GG}{{\mathbb G}}
\nc{\HH}{{\mathbb H}}   \nc{\LL}{{\mathbb L}}
\nc{\NN}{{\mathbb N}}   \nc{\PP}{{\mathbb P}}
\nc{\QQ}{{\mathbb Q}}   \nc{\RR}{{\mathbb R}}
\nc{\TT}{{\mathbb T}}   \nc{\VV}{{\mathbb V}}
\nc{\ZZ}{{\mathbb Z}}   \nc{\TP}{\widetilde{P}}
\nc{\m}{{\mathbbm m}}

\nc{\cala}{{\mathcal A}}    \nc{\calc}{{\mathcal C}}
\nc{\cald}{\mathcal{D}}     \nc{\cale}{{\mathcal E}}
\nc{\calf}{{\mathcal F}}    \nc{\calg}{{\mathcal G}}
\nc{\calh}{{\mathcal H}}    \nc{\cali}{{\mathcal I}}
\nc{\call}{{\mathcal L}}    \nc{\calm}{{\mathcal M}}
\nc{\caln}{{\mathcal N}}    \nc{\calo}{{\mathcal O}}
\nc{\calp}{{\mathcal P}}    \nc{\calr}{{\mathcal R}}
\nc{\cals}{{\mathcal S}}    \nc{\calt}{{\Omega}}
\nc{\calv}{{\mathcal V}}    \nc{\calw}{{\mathcal W}}
\nc{\calx}{{\mathcal X}}

\nc{\fraka}{{\mathfrak a}}
\nc{\frakb}{\mathfrak{b}}
\nc{\frakg}{{\frak g}}
\nc{\frakl}{{\frak l}}
\nc{\fraks}{{\frak s}}
\nc{\frakB}{{\frak B}}
\nc{\frakm}{{\frak m}}
\nc{\frakM}{{\frak M}}
\nc{\frakp}{{\frak p}}
\nc{\frakW}{{\frak W}}
\nc{\frakX}{{\frak X}}
\nc{\frakS}{{\frak S}}
\nc{\frakA}{{\frak A}}
\nc{\frakx}{{\frakx}}

\nc{\lir}[1]{\textcolor{red}{\underline{Li:}#1 }}

\begin{document}

\title[Differential algebras]{Cohomologies, extensions and deformations of differential algebras with any weights}

\author{Li Guo}
\address{
Department of Mathematics and Computer Science,
         Rutgers University,
         Newark, NJ 07102}
\email{liguo@rutgers.edu}
\author{Yunnan Li }
\address{School of Mathematics and Information Science, Guangzhou University,   Guangzhou 510006, China}
\email{ynli@gzhu.edu.cn}
\author{Yunhe Sheng}
\address{Department of Mathematics, Jilin University, Changchun 130012, Jilin, China}
\email{shengyh@jlu.edu.cn}

\author{Guodong Zhou}
\address{  School of Mathematical Sciences, Shanghai Key Laboratory of PMMP,
  East China Normal University,
 Shanghai 200241,
   China}
\email{gdzhou@math.ecnu.edu.cn}

\date{\today}

\begin{abstract}
As an algebraic study of differential equations, differential algebras have been studied for a century and and become an important area of mathematics. In recent years the area has been expended to the noncommutative associative and Lie algebra contexts and to the case when the operator identity has a weight in order to include difference operators and difference algebras. This paper provides a cohomology theory for differential algebras of any weights.  This gives a uniform approach to both the zero weight case which is similar to the earlier study of differential Lie algebras, and the non-zero weight case which poses new challenges. As applications, abelian extensions of a differential algebra are classified by the second cohomology group. Furthermore, formal  deformations of differential algebras are obtained and the rigidity of a differential algebra is characterized by the vanishing of the second cohomology group.
\end{abstract}

\subjclass[2010]{
16E40   
16S80   
12H05   
12H10   
16W25   
16S70   
}

\keywords{cohomology, extension, deformation, differential algebra, difference algebra, derivation}

\maketitle

\tableofcontents

\allowdisplaybreaks

\section{Introduction}

This paper studies the cohomology theory, abelian extensions and formal deformations for differential algebras of any weights.

\subsection{Differential algebras, old and new}
Classically, a differential algebra is a commutative algebra equipped with a linear operator satisfying the Leibniz rule, modeled after the differential operator in analysis. In fact, the origin of differential algebras is the algebraic study of differential equations pioneered by Ritt in the 1930s~\cite{Rit,Rit1}. Through the work of many mathematicians~\cite{Kol,Mag,PS} in the following decades, the subject has been fully developed into a vast area in mathematics, comprising of differential Galois theory, differential algebraic geometry and differential algebraic groups, with broad connections to other areas in mathematics such as arithmetic geometry and logic, as well as computer science (mechanical proof of geometric theorems) and mathematical physics (renormalization in quantum field theory)~\cite{CM,FLS,MS,Wu,Wu2}.

Another broadly used notion of differential operators is the derivations on (co)chain complexes. There the operator $d$ is assumed to satisfy the nilpotent condition $d^2=0$, in which case a differential algebra induces an associative diassociative algebra structure~\mcite{Lo01}, in analog to the case of Rota-Baxter algebras that induce dendriform algebras and tridendriform algebras.

In recent years, differential algebras without the commutative or nilpotent conditions have been considered, to include naturally arisen algebras such as path algebras and to have a more meaningful differential Lie algebra theory generalizing the classical relationship between associative algebras and Lie algebras~\cite{GL,Poi,Poi1}. In~\mcite{Loday} differential algebra is studied from an operadic point of view. Differential algebras have also been applied to control theory and gauge theory through the BV-formalism~\mcite{AKT,Va}.
In~\cite{ATW}, a notion of differential algebras were generalized to non(anti)commutative superspace by deformation in the study of instantons in string theory.

In another direction, the Leibniz rule is generalized to include the difference quotient $\frac{f(x+\lambda)-f(x)}{\lambda}$ before taking the limit $\lambda\mapsto 0$, leading to the notion of a differential algebra of weight $\lambda$~\cite{GK}. This generalized notion of differential algebra provides a framework for a uniform approach of the differential algebra (corresponding to the case when $\lambda=0$ and another extensively studied algebraic structure, the difference algebra (corresponding to the case when $\lambda=1$)~\cite{Coh,Lev,PS1}, as an algebraic study of difference equations. This notion also furnishes an algebraic context for the study of quantum calculus~\mcite{KC}. Differential operators with weights on Virasoro algebras were also investigated~\mcite{LGG}.

\subsection{Homology and deformations of differential algebras}
As noted above, nilpotent differential operators are fundamental notions for complexes to define cohomology which in turn plays a key role in deformation theory, either for specific algebraic structures, starting with the seminal works of Gerstenhaber for associative algebras and of Nijenhuis and Richardson for Lie algebras~\mcite{Ge0,Ge,NR}, or for the general context of operads, culminated in the monographs~\mcite{MSS,LV}.
As a further step in this direction, studies of deformations and the related cohomology have recently emerged for algebras with linear operators, including Rota-Baxter operators  and differential operators on Lie algebras~\mcite{TBGS,TFS}.

The importance of differential (associative) algebras makes it compelling to develop their cohomology theory, in both the zero weight case and nonzero weight case. The natural role played by nilpotent differential operators in cohomology makes it even more fascinating in this study. The purpose of this paper is to develop such a theory for differential (associative) algebras of any weight, and give its applications in the study of abelian extensions and formal deformations of differential algebras.

In comparison with the recent work~\mcite{TFS} on cohomology and deformation of differential Lie algebras, we note that the derivations there are of weight zero in which case its approach can be adapted for differential (associative) algebras. Our main emphasis on differential algebras in this paper is for the nonzero weight case, which is needed in order to study difference operators and difference algebras, but for which a different approach has to be taken. See comments in the outline below and Remark~\mref{rk:nonzero}.

\subsection{Layout of the paper}

The paper is organized as follows. In Section \ref{sec:bimod}, we introduce the notion of a bimodule over a differential algebra of nonzero weight, and provide its characterization in terms of a monoid object in slice categories.

A differential algebra is the combination of the underlying algebra and the differential operator. In this light we build the cohomology theory of a differential algebra by combining its components from the algebra and from the differential operator.
Thus in Section  \ref{sec:cohomologydo}, we establish the cohomology theory for differential operators of any weights, which is quite different from the one for the underlying algebra unless the weight is zero.
In Section \ref{sec:cohomologydasub}, we combine the Hochschild cohomology for associative algebras and the just established cohomology for differential operators of any weights to define the cohomology of differential algebras of any weights, with the cochain maps again posing extra challenges when the weight is not zero.
Finally in Section \ref{sec:cohomologyrelation}, we establish a close relationship among these cohomologies.
More precisely, we show that there is a short exact sequence of cochain complexes for the algebra, the differential operator and the differential algebra. The resulting long exact sequence gives linear maps from the cohomology groups of the differential algebra to those of the algebra, with the error terms (kernels and cokernels) controlled by the cohomology groups of the differential operator.

As applications and further justification of our cohomology theory for differential algebras, in Section \ref{sec:ext}, we apply the theory to study abelian extensions of differential algebras of any weights, and show that abelian extensions are classified by the second cohomology group of the differential algebras.

Further, in Section \ref{sec:def}, we apply the above cohomology theory to study formal deformations of differential algebras of any weights. In particular, we show that if  the second cohomology group of a differential algebra with coefficients in the regular representation is trivial, then this differential algebra is rigid.

\smallskip

\noindent
{\bf Notation.}
Throughout this paper, $\bk$ denotes a field of characteristic zero. All the vector spaces, algebras, linear maps and tensor products are taken over $\bk$ unless otherwise specified.

\section{Differential algebras and their bimodules} \label{sec:bimod}
This section gives background on differential algebras and first results on their bimodules,  with an interpretation in the general context of monoid objects in slice categories~\mcite{DMZ}.

\subsection{The category of bimodules over differential algebras}

\begin{defn} (\cite{GK})
	Let $\lambda\in \bk$ be a  fixed element. A  {\bf differential algebra of weight $\lambda$} (also called a  {\bf $\lambda$-differential algebra}) is an associative   algebra $A$ together with a  linear operator $d_A : A\rar A$ such that
	\begin{equation}\label{diff}
	d_A(xy)=d_A(x)y+x d_A(y)+\lambda d_A(x)d_A(y), \quad\forall x,y\in A.
	\end{equation}
	If $A$ is unital, it further requires that
	\begin{equation}\label{difu}
	d_A(1_A)=0.
	\end{equation}
	Such an operator is called a {\bf differential operator of weight $\lambda$} or a  {\bf derivation of weight $\lambda$}. It is also called a  {\bf $\lambda$-differential operator} or a  {\bf $\lambda$-derivation}.

Given two differential algebras $(A, d_A),\,(B,d_B)$ of the same  weight $\lambda$, a {\bf homomorphism of differential algebras} from $(A,d_A)$ to $(B,d_B)$ is an algebra homomorphism $\varphi:A\lon B$ such that $\varphi\circ d_A=d_B\circ\varphi$.
We denote by $\Diffl$ the category of $\lambda$-differential  algebras.
\end{defn}

To simply notations, for all the above notions, we will often suppress the mentioning of the weight $\lambda$ unless it needs to be specified.

Recall that a bimodule of an associative algebra $A$ is a triple $(V,\rho_l,\rho_r)$, where $V$ is a vector space, $\rho_l: A\to \mathrm{End}_\bk(V),~ x\mapsto (v\mapsto xv)$ and $\rho_r: A\to \mathrm{End}_\bk(V),~ x\mapsto (v\mapsto vx)$ are homomorphism and anti-homomorphism of associative algebras respectively such that $(xv)y=x(vy)$ for all $x,y\in A$ and $v\in V$.

\begin{defn}
 Let $(A, d_A)$ be a differential algebra.
 \begin{itemize}
   \item[{\rm (i)}] A {\bf bimodule} over the differential algebra $(A, d_A)$ is a quadruple $(V,\rho_l,\rho_r,d_V)$, where $d_V\in \mathrm{End}_\bk(V)$, and  $(V,\rho_l,\rho_r)$ is a bimodule over the associative algebra $A$, such that for all $x,y\in A, v\in V,$
   the following equalities hold:
   \begin{eqnarray*}
     d_V(xv)&=&d_A(x)v+xd_V(v)+\lambda d_A(x)d_V(v),\\
      d_V(vx)&=&v d_A(x)+d_V(v)x+\lambda d_V(v)d_A(x).
   \end{eqnarray*}

   \item[{\rm (ii)}] Given two bimodules $(U,\rho^U_l,\rho^U_r, d_U),\,(V,\rho^V_l,\rho^V_r, d_V)$ over  $(A, d_A)$, a  linear map $f:U\lon V$ is called a {\bf homomorphism} of bimodules, if $f\circ d_U=d_V\circ f$ and $$f\circ\rho^U_l(x)=\rho^V_l(x)\circ f,\quad f\circ\rho^U_r(x)=\rho^V_r(x)\circ f,\quad \forall x\in A.$$
 \end{itemize}
\end{defn}

We denote by  $(A,d_A)\bim$  the category of bimodules over the differential algebra $(A, d_A)$.

\begin{exam}
Any differential algebra $(A, d_A)$ is a bimodule over itself with
$$\rho_l:A\to \mathrm{End}_\bk(A),\,x\mapsto(y\mapsto xy),\quad \rho_r: A\to \mathrm{End}_\bk(A),\,x\mapsto(y\mapsto yx).$$ It is called the {\bf regular bimodule} over the differential algebra $(A, d_A)$.
\end{exam}

It is straightforward to obtain the following result.

\begin{prop}
  Let $(V,\rho_l,\rho_r, d_V)$ be a bimodule of the  differential algebra $(A, d_A)$. Then $(A\oplus V, d_A\oplus d_V)$ is a differential algebra, where the associative algebra structure on $A\oplus V$ is given by
  $$
  (x+u)(y+v)=xy+xv+uy,\quad \forall x,y\in A,~u,v\in V.
  $$
\end{prop}

\subsection{Differential bimodules in terms of monoid objects in slice categories}
We now show that the above definition of bimodules for differential algebras coincides with the notion obtained by applying monoid objects in certain slice categories.

\subsubsection{Monoid objects in slice categories}
We first recall some general concepts~\mcite{DMZ}.
\begin{defn}
For a category $\calc$ and an object $A$ in $\calc$. The {\bf slice category} $\calc/A$ is the category whose
\begin{itemize}
\item objects $(B,\pi)$ are $\calc$-morphisms $\pi:B\,\rightarrow\, A,B\in\calc$, and
\item morphisms $(B',\pi')\stackrel{f}{\,\rightarrow\,}(B'',\pi'')$ are commutative diagrams of $\calc$-morphisms:
\[\begin{CD}
B'@>f>>B''\\
  @V\pi' VV   @VV\pi''V  \\
A@=A.
\end{CD}\]
\end{itemize}
\end{defn}

\begin{defn}
 Let $\calc$ be a category with finite products and a terminal object $T$. A {\bf monoid} object in $\calc$ is an object $X\in \ob(\calc)$ together with two morphisms $\mu:X\times X\lon X$ and $\eta:T\lon X$ such that following diagrams commute:
\begin{itemize}
\item the {\bf associativity} of $\mu:$
\[\begin{CD}\label{neut}
X\times X\times X@>\mu\times {\Id}_X>>X\times X\\
  @V{\Id}_X\times \mu VV   @VV\mu V  \\
X\times X@>>\mu >X,
\end{CD}\]
\item the {\bf neutrality} of $\eta:$
\[\begin{CD}
X\times X@>\mu>>           X        @<\mu<<                             X\times X        \\
  @A{\Id}_X\times \eta AA    @|                                   @AA\eta\times {\Id}_X A    \\
X\times T@<<({{\Id}_X},t_X) < X        @>>(t_X,{{\Id}_X})>                    T\times X,
\end{CD}\]
\end{itemize}
where $t_X:X\lon T$ is the unique morphism.

Let $\calc_\m$ be the category whose objects are monoid objects $(X,\mu,\eta)$ in $\calc$ as above and the hom-set $\Hom_{\calc_\m}((X,\mu,\eta),(X',\mu',\eta'))$ is the set of all $f\in\Hom_{\calc}(X,X')$ for which $\mu'\circ (f\times f)=f\circ \mu$ and $\eta'=f\circ\eta$.
\end{defn}

\subsubsection{The differential algebra case}
Fix a  differential  algebra $(A,d)$, and consider the slice category $\Diffl/A$. The terminal object in $\Diffl/A$ is $T=A\stackrel{\Id}{\lon}A$. Given
$(A_1,d_1),\,(A_2,d_2)\in\Diffl$ and $X_1=(A_1,\varphi_1),\,X_2=(A_2,\varphi_2)\in\Diffl/A$, the product $X_1\times X_2$ is given by $(A_1\times_A A_2,\bar{\varphi})$, where
\[A_1\times_A A_2:=\{(a,a')\in A_1\times A_2\,|\,\varphi_1(a)=\varphi_2(a')\},\]
and
\[\bar{\varphi}:A_1\times_A A_2\lon A,\quad (a,a')\mapsto \varphi_1(a).\]
For any $(A',d')\in\Diffl$, $X=A'\stackrel{\varphi}{\lon}A\in\Diffl/A$ is
a monoid object if and only if there exist differential algebra homomorphisms $M:A'\times_A A'\lon A'$ and $\iota:A\lon A'$ such that
\begin{align*}
 \bar{\varphi}=\varphi M,\quad \varphi\iota=\Id_A,\quad M(M(a,b),c)=M(a,M(b,c)),\quad M(\iota(\varphi(a)),b)=b,\quad M(a,\iota(\varphi(b)))=a.
\end{align*}
for any $(a,b,c)\in A'\times_A A'\times_A A'$. Consequently,
\begin{align*}
M(a,a')&=M(a-\iota(\varphi(a)),0)+M(0,a'-\iota(\varphi(a')))
+M(\iota(\varphi(a)),\iota(\varphi(a')))\\
&=a-\iota(\varphi(a))+a'-\iota(\varphi(a'))
+\iota(\varphi(a'))\\
&=a+a'-\iota(\varphi(a)),
\end{align*}
for any $(a,a')\in A'\times_A A'$. Let $V:=\ker\varphi$ as a differential ideal of $A'$. Then $A'=\iota(A)\oplus V$. Since $M$ is also a differential algebra homomorphism, we have
\[uv=M(u,0)M(0,v)=M((u,0)(0,v))=M(0,0)=0,\]
for any $u,v\in V$. On the other hand, as $\iota d=d'\iota$ and $\varphi d'=d\varphi$, it is clear that $V$ is an $A$-bimodule with differential $d_V:=d'|_V$ by letting
\[\rho_l(x)v=\iota(x)v,\,\rho_r(x)v=v\iota(x),\]
for any $x\in A,\,v\in V$.

Conversely, given any bimodule $(V,d_V)$ of differential algebra $(A,d_A)$, one can define a differential algebra structure on $A\oplus V$ naturally with differential operator $(d_A,d_V)$ by letting
\[(x,u)(y,v)=(xy,xv+uy),\quad \forall x,y\in A,\,u,v\in V.\]
 We say this differential algebra the {\bf semi-direct product} of $A$ and $V$, and denote it as $(A\ltimes V,d_\ltimes)$.
Let $p:A\ltimes V\lon A$ be the canonical projection, and associate it with
\[(A\ltimes V)\times_A(A\ltimes V)=\{((x,u),(x,v))\in (A\ltimes V)\times(A\ltimes V)\,|\,x\in A,\,u,v\in V\}.\]
Define differential algebra homomorphisms $$M_\ltimes:(A\ltimes V)\times_A(A\ltimes V)\lon A\ltimes V,\,((x,u),(x,v))\mapsto(x,u+v)$$
and
$$i_\ltimes:A\lon A\ltimes V,\,x\mapsto(x,0).$$
Then the following lemma is easy to check.
\begin{lemma}
For any $A$-bimodule $V$, $X=A\ltimes V\stackrel{p}{\lon}A$ with morphisms $\mu_\ltimes,\eta_\ltimes$ determined by $M_\ltimes,i_\ltimes$ respectively is a monoid object in $\Diffl/A$.
\end{lemma}

\begin{theorem}
 The functors $\xymatrix@=2em{(\Diffl/A)_\m \ar@/^/[r]^-{\ker} &(A,d_A)\bim\ar@/^/[l]^-{\ltimes} }$ induce an equivalence of categories.
 \mlabel{thm:monoid}
\end{theorem}
\begin{proof}
First note that $\ker\circ\,\ltimes=\Id_{(A,d_A)\bim}$. Define $\Psi:\Id_{(\Diffl/A)_\m}\lon\ltimes\circ\ker$ by letting
\[\Psi_X:(A'\stackrel{\varphi}{\lon}A,M,\iota)\lon (A\ltimes\ker \varphi\stackrel{p}{\lon}A,M_\ltimes,i_\ltimes),\quad a\mapsto(\varphi(a),a-\iota(\varphi(a))).\]
for any $X=A'\stackrel{\varphi}{\lon}A\in(\Diffl/A)_\m$.

By definition $\varphi=p\circ\Psi_X$, thus $\Psi_X$ is a morphism in $(\Diffl/A)_\m$. The inverse of $\Psi_X$ is
\[\Psi_X^{-1}:(A\ltimes\ker \varphi\stackrel{p}{\lon}A,M_\ltimes,i_\ltimes)\lon(A'\stackrel{\varphi}{\lon}A,M,\iota),\quad (x,u)\mapsto u+\iota(x).\]
Also, for any $(a,a')\in A'\times_A A'$, $x\in A$, we have
\begin{eqnarray*}
M_\ltimes(\Psi_X(a),\Psi_X(a'))
&=&M_\ltimes((\varphi(a),a-\iota(\varphi(a))),(\varphi(a'),a'-\iota(\varphi(a'))))\\
&=&(\varphi(a),a+a'-2\iota(\varphi(a)))\\
&=&\Psi_X(a+a'-\iota(\varphi(a)))=\Psi_X(M(a,a')),\\
\Psi_X(\iota(x))&=&(\iota(x),\iota(x)-\iota(\varphi(\iota(x))))=(\iota(x),0)=i_\ltimes(x),\\
\Psi_X(d'(a))&=&(\varphi(d'(a)),d'(a)-\iota(\varphi(d'(a))))=(d(\varphi(a)),d'(a-\iota(\varphi(a))))\\
&=&(d,d')\Psi_X(a).
\end{eqnarray*}
For any $X_1=A_1\stackrel{\varphi_1}{\lon}A,\,X_2=A_2\stackrel{\varphi_2}{\lon}A\in(\Diffl/A)_\m$ and morphism $f:X_1\lon X_2$, we have
\begin{eqnarray*}
\Psi_{X_2} (f(a))&=&(\varphi_2(f(a)),f(a)-\iota_2(\varphi_2(f(a))))\\
&=&(\varphi_1(a),f(a)-\iota_2(\varphi_1(a)))\\
&=&(\varphi_1(a),f(a-\iota_1(\varphi_1(a))))\\
&=&(\Id_A,f)(\Psi_{X_1}(a)).
\end{eqnarray*}
Hence,  $\Psi$ is a natural isomorphism. The categories $(\Diffl/A)_\m$ and $(A,d_A)\bim$ are equivalent to each other.
\end{proof}

\section{Cohomology of differential algebras}\label{sec:cohomologyda}

 Let $V$ be a bimodule of an associative algebra $A$. Denote by $C^n_\alg(A,V)=\Hom(\otimes^n A,V)$. In particular, $C^0_\alg(A,V)=V.$ Recall that
  the  Hochschild cochain complex is the cochain complex $(C^*_\alg(A,V)=\oplus_{n=0}^\infty C^n_\alg(A,V),\partial)$, where the coboundary operator $$\partial: C^n_\alg(A, V)\longrightarrow  C^{n+1}_\alg(A, V)$$ is given by
\[\begin{split}
\partial  f(x_1,\dots,x_{n+1})&=x_1 f(x_2,\dots,x_{n+1})+\sum_{i=1}^n(-1)^if(x_1,\dots,x_ix_{i+1},\dots,x_{n+1})\\
&+(-1)^{n+1}f(x_1,\dots,x_n) x_{n+1}
\end{split}\]
 for all $f\in C^n_\alg(A, V),~x_1,\dots,x_{n+1}\in A$.  The corresponding  Hochschild cohomology is denoted by  $HH^*_\alg(A, V)$

 \subsection{Cohomology of differential operators}\label{sec:cohomologydo}

 Let $(A,d_A)$ be a differential algebra of weight $\lambda$  and let $(V, d_V)$ be a bimodule over $(A, d_A)$. In this subsection, we define the cohomology of differential operators.

First we make the following observation, noting that the bimodule structure coincides with the regular bimodule when the weight $\lambda$ is zero.
\begin{lemma}\label{le:bmd}
A differential algebra $(A,d_A)$  admits  a new bimodule structure on  $(V,d_V)$ given by:
\[x\vdash_\lambda v=(x+\lambda d_A(x))v,\quad v\dashv_\lambda x=v(x+\lambda d_A(x)),\quad \forall x \in A, v\in V.\]
\end{lemma}
\begin{proof}

Given $x,y\in A$ and $v\in V$, we have
\[\begin{split}
x\vdash_\lambda (y\vdash_\lambda v)&=(x+\lambda d_A(x))((y+\lambda d_A(y))v)\\
&=(xy+\lambda(xd_A(y)+d_A(x)y+\lambda d_A(x)d_A(y)))v\\
&=(xy+\lambda d_A(xy))v\\
&=(xy)\vdash_\lambda v.
\end{split}\]
Similarly, $(v\dashv_\lambda x)\dashv_\lambda y=v\dashv_\lambda (xy)$. Thus $(V,\vdash_\lambda,\dashv_\lambda)$ is a bimodule over the associative algebra $A$.

For $x\in A$ and $v\in V$,
\[\begin{split}
d_V(x\vdash_\lambda v)&=d_V((x+\lambda d_A(x))v)\\
&=d_V(xv)+\lambda d_V(d_A(x)v)\\
&=d_A(x)v+xd_V(v)+\lambda d_A(x)d_V(v)+\lambda (d_A(d_A(x))v+d_A(x)d_V(v)+\lambda d_A(d_A(x))d_V(v))\\
&=d_A(x)v+\lambda d_A(d_A(x))v+xd_V(v)+\lambda d_A(x)d_V(v)+\lambda  ( d_A(x)d_V(v)+\lambda d_A(d_A(x))d_V(v))\\
&=  d_A(x)\vdash_\lambda v+x\vdash_\lambda d_V(v)+\lambda d_A(x)\vdash_\lambda d_V(v).
\end{split}\]
Similarly, one shows the equality $$d_V( v\dashv_\lambda x)=v \dashv_\lambda d_A(x)+ d_V(v)\dashv_\lambda x+\lambda d_V(v) \dashv_\lambda d_A(x).$$  Also, it is obvious that $(x\vdash_\lambda v)\dashv_\lambda y=x\vdash_\lambda (v\dashv_\lambda y)$. Thus, $(V,\vdash_\lambda,\dashv_\lambda,d_V)$ is a bimodule over the differential algebra $(A,d_A)$.
\end{proof}	

For distinction, we let $V_\lambda$ denote the new  bimodule structure over $(A, d_A)$  given in Lemma~\ref{le:bmd}.

Denote by $C^n_\diffo(d_A,d_V)=\Hom(\otimes^n A,V)$, which is called the space of $n$-chains of the differential operator $d_A$ with coefficients in the bimodule $(V,d_V)$.
\begin{defn}The cohomology of the cochain complex $(C^*_\diffo(d_A,d_V)=\oplus_{n=0}^\infty C^n_\diffo(d_A,d_V),\partial_\lambda)$, denoted by $H^*_\diffo(d_A,d_V)$, is called the {\bf cohomology of the differential operator} $d_A$ with coefficients in the bimodule $(V,d_V)$, where
  $$\partial_\lambda:C^n_\diffo(d_A,d_V)\longrightarrow C^{n+1}_\diffo(d_A,d_V)$$ is the Hochschild coboundary operator of the associative algebra $A$ with the coefficients in the bimodule $V_\lambda$ given in Lemma \ref{le:bmd}. More precisely, we have
\[\begin{split}
\partial_\lambda f(x_1,\dots,x_{n+1})&=x_1\vdash_\lambda f(x_2,\dots,x_{n+1})+\sum_{i=1}^n(-1)^if(x_1,\dots,x_ix_{i+1},\dots,x_{n+1})\\
&+(-1)^{n+1}f(x_1,\dots,x_n)\dashv_\lambda x_{n+1}
\end{split}\]
for all $f\in C^n_\diffo(d_A, d_V),~x_1,\dots,x_{n+1}\in A$.
\end{defn}

\subsection{Cohomology of differential algebras} \label{sec:cohomologydasub}

We now combine the classical Hochschild cohomology of associative algebras and the newly defined cohomology of differential operators  to  define the cohomology of the differential algebra $(A,d_A)$ with coefficients in the bimodule $(V,d_V)$.

Define the set of $n$-cochains by
\begin{equation}\label{eq:dac}
C_{\Diffl}^n(A,V):=
\begin{cases}
C^n_\alg(A,V)\oplus C^{n-1}_\diffo(d_A,d_V),&n\geq1,\\
C^0_\alg(A,V)=V,&n=0.
\end{cases}
\end{equation}
Define a linear map $\partial_{\Diffl}:C_{\Diffl}^n(A,V)\rar C_{\Diffl}^{n+1}(A,V)$ by
\begin{eqnarray}\label{eq:pda}
\partial_{\Diffl}(f,g)&:=&(\partial f,\partial_\lambda g+(-1)^n\delta f),\quad \forall f\in C^n_\alg(A,V),\,g\in C^{n-1}_\diffo(d_A,d_V),\quad n\geq 1,\\
\label{eq:pda2}\partial_{\Diffl} {v}&:=&(\partial {v}, \delta v),\quad\forall {v}\in C^0_\alg(A,V)=V,
\end{eqnarray}
where the   linear map $\delta:C^n_\alg(A,V)\rar C^n_\diffo(d_A,d_V)$ is defined by
\begin{align*}
 \delta f(x_1,\dots,x_n):=\sum_{k=1}^n\lambda^{k-1}\sum_{1\leq i_1<\cdots<i_k\leq n}f(x_1,\dots,d_A(x_{i_1}),\dots,d_A(x_{i_k}),\dots,x_n)-d_V f(x_1,\dots,x_n),
\end{align*}
for any $f\in C^n_\alg(A,V)$, $  n\geq 1$ and $$\delta v= {-} d_V {(v)},\quad \forall v\in C^0_\alg(A,V)=V.$$

\begin{prop}\mlabel{prop:delta}
The linear map $\delta$ is a cochain map from the cochain complex $(C^*_\alg(A,V),\partial)$ to $(C^*_\diffo(d_A,d_V),\partial_\lambda)$.
\end{prop}

The rather long and technical proof of this result is is postponed to the appendix in order not to interrupt the flow of the presentation.

\begin{remark}
Note that $C^n_\diffo(d_A,d_V)$ equals to $C^n_\alg(A,V)$ as linear spaces but they are equal as cochain complexes only when $\lambda=0$. When $\lambda$ is not zero, a new bimodule structure is needed to define $\partial_\lambda$ which eventually leads to the rather long and technical argument in order to establish the cochain map in Proposition~\mref{prop:delta}.
\mlabel{rk:nonzero}
\end{remark}

\begin{theorem}\label{thm: cochain complex for differential algebras}
The pair $(C_{\Diffl}^*(A, V),\partial_{\Diffl})$ is a cochain complex. So $\partial_{\Diffl}^2=0.$
\end{theorem}
\begin{proof}
For any $v\in C^0_\alg(A, V)$, by Proposition~\mref{prop:delta} we have
\[\partial_{\Diffl}^2 v=\partial_{\Diffl}(\partial v,\delta v)=(\partial^2 v,\partial_\lambda\delta v-\delta\partial v)=0.\]
Given any $f\in C^n_\alg(A,V),\,g\in C^{n-1}_\diffo(d_A,d_V)$ with $n\geq1$, also  by Proposition~\mref{prop:delta} we have
\[\partial_{\Diffl}^2(f,g)=\partial_{\Diffl}(\partial f,\partial_\lambda g+(-1)^n\delta f)=(\partial^2 f,\partial_\lambda(\partial_\lambda g+(-1)^n\delta f)+(-1)^{n+1}\delta\partial f)=0.\]
Therefore,   $(C_{\Diffl}^*(A, V),\partial_{\Diffl})$ is a cochain complex.
\end{proof}

\begin{defn}
  The cohomology of the cochain complex $(C_{\Diffl}^*(A, V),\partial_{\Diffl})$, denoted by $H_{\Diffl}^*(A, V)$, is called the {\bf cohomology of the differential algebra} $(A,d_A)$ with coefficients in the bimodule $(V,d_V)$.
\end{defn}

To end this subsection, we compute 0-cocycles, 1-cocycles and 2-cocycles of the cochain complex $(C_{\Diffl}^*(A,V),\partial_{\Diffl})$.

It is obvious that for all $v\in V$, $\partial_{\Diffl}v=0$ if and only if
$$\partial v=0,\quad d_V(v)=0.$$

For all $(f,v)\in\Hom (A,V)\oplus V$,  $\partial_{\Diffl}(f,v)=0$ if and only if $\partial f=0,$ and
$$
x\vdash_\lambda v-v\dashv_\lambda x=f(d_A(x))-d_V(f(x)),\quad \forall x\in A.
$$

For all $(f,g)\in\Hom (\otimes^2A,V)\oplus\Hom (A,V) $,  $\partial_{\Diffl}(f,g)=0$ if and only if $\partial f=0,$ and
\begin{align*}
 x\vdash_\lambda g(y)-g(xy)+g(x)\dashv_\lambda y=  -\lambda f(d_A(x),d(y))-f(d_A(x),y)-f(x,d_A(y))+d_V(f(x,y)),\
\end{align*}
for all $x,y\in A.$

 In the next two sections, we shall need a subcomplex of the cochain complex $C_{\Diffl}^*(A, V)$. Let
\begin{equation}\label{eq:dacsub}
\tilde{C}_{\Diffl}^n(A,V):=
\begin{cases}
 C^n_\alg(A,V)\oplus C^{n-1}_\diffo(d_A,d_V),&n\geq2,\\
C^1_\alg(A,V),&n=1,\\
0,&n=0.
\end{cases}
\end{equation}
Then it is obvious that $(\tilde{C}_{\Diffl}^*(A,V)=\oplus_{n=0}^\infty\tilde{C}_{\Diffl}^n(A,V),\partial_{\Diffl} )$  is a subcomplex of the cochain complex  $(C_{\Diffl}^*(A, V),\partial_{\Diffl})$. We  denote its cohomology by $\tilde{H}_{\Diffl}^*(A,V)$. Obviously, $\tilde{H}_{\Diffl}^n(A,V)= {H}_{\Diffl}^n(A,V)$ for $n>2$.

\subsection{Relationship among the cohomologies}\label{sec:cohomologyrelation}

The coboundary operator $\partial_{\Diffl}$ can be illustrated by the following diagram:
 \[
\small{ \xymatrix{
\cdots
\longrightarrow C^n_\alg(A,V)\ar[dr]^{(-1)^n\delta} \ar[r]^{\qquad\partial} & C^{n+1}_\alg(A,V) \ar[dr]^{(-1)^{n+1}\delta} \ar[r]^{\partial\qquad}  &C^{n+2}_\alg(A,V)\longrightarrow\cdots  \\
\cdots\longrightarrow C^{n-1}_\diffo(d_A,d_V) \ar[r]^{\qquad\partial_\lambda} &C^{n}_\diffo(d_A,d_V)\ar[r]^{\partial_\lambda\qquad}&C^{n+1}_\diffo(d_A,d_V)\longrightarrow \cdots.}
}
\]

Thus we have
 \begin{prop}\label{pro:exactsequence}
There exists an exact sequence of cochain complexes,
\[0\rar C^{*-1}_\diffo(d_A, d_V)\stackrel{\iota}{\rar} C_{\Diffl}^*(A, V)\stackrel{\pi}{\rar} C^*_\alg(A, V)\rar 0,\]
where $\iota$ and $\pi$ are the   inclusion  and the projection respectively.
\end{prop}

The relations among the various cohomology groups are given by the following theorem.

\begin{theorem}
  We have the following long exact sequence of cohomology groups,
\[\cdots \rar H^{n-1}_\diffo(d_A, d_V)\stackrel{\bar{\iota}}{\rar} H_{\Diffl}^n(A, V)\stackrel{\bar{\pi}}{\rar} HH^n_\alg(A, V)\stackrel{(-1)^n\bar{\delta}}{\rar} H^n_\diffo(d_A, d_V)\rar \cdots,\]
where $\bar{\delta} :HH^n_\alg(A, V)\rar H^n_\diffo(d_A, d_V)$ is given by $\bar{\delta}[f]=[\delta f]$. Here $[f]$ and $ [\delta f]$ denote  the cohomological classes of $f\in C^n_\alg(A,V)$ and  $\delta f\in C^n_\diffo(d_A,d_V)$.
\end{theorem}

Thus the linear maps $\bar{\pi}$ establish a relationship between the cohomology groups of the differential algebra and those of the underlying algebra, with the error terms controlled by the cohomology groups of the differential operator. This is resemblance of the Mayer-Vietoris sequence.

\begin{proof}
By Proposition \ref{pro:exactsequence} and the Snake Lemma, we have the long exact sequence
\[\cdots \rar H^{n-1}_\diffo(d_A, d_V)\stackrel{\bar{\iota}}{\rar} H_{\Diffl}^n(A, V)\stackrel{\bar{\pi}}{\rar} HH^n_\alg(A, V)\stackrel{\Delta_n}{\rar} H^n_\diffo(d_A, d_V)\rar \cdots.\]
It remains to prove that the connecting homomorphism $\Delta_n:HH^n_\alg(A, V)\rar H^n_\diffo(d_A, d_V)$ are exactly $(-1)^n\bar{\delta} $. Indeed, by the construction of $\Delta_n$ and Eq.~\eqref{eq:pda}, for any $[f]\in HH^n_\alg(A, V)$, we have
\[\bar{\iota}\Delta_n([f])=[\partial_{\Diffl}(f,0)]
= [(0,(-1)^n\delta f)],\]
which implies that $\Delta_n=(-1)^n\bar{\delta}$.
\end{proof}

\section{Abelian extensions of differential algebras} \label{sec:ext}
In this section, we study abelian extensions of differential algebras and show that they are classified by the second cohomology, as one would expect of a good cohomology theory.

\begin{defn}
An {\bf abelian extension} of differential algebras is a short exact sequence of homomorphisms of differential algebras
\[\begin{CD}
0@>>> {V} @>i >> \hat{A} @>p >> A @>>>0\\
@. @V {d_V} VV @V d_{\bar{A}} VV @V d_A VV @.\\
0@>>> {V} @>i >> \hat{A} @>p >> A @>>>0
\end{CD}\]
such that $uv=0$ for all $u,v\in V.$
\end{defn}

We will call $(\hat{A},d_{\hat{A}})$ an abelian extension of $(A,d_A)$ by $(V,d_V)$.

\begin{defn}
Let $(\hat{A}_1,d_{\hat{A}_1})$ and $(\hat{A}_2,d_{\hat{A}_2})$ be two abelian extensions of $(A,d_A)$ by $(V,d_V)$. They are said to be {\bf isomorphic} if there exists an isomorphism of differential algebras $\zeta:(\hat{A}_1,d_{\hat{A}_1})\rar (\hat{A}_2,d_{\hat{A}_2})$ such that the following commutative diagram holds:
\[\begin{CD}
0@>>> {(V,d_V)} @>i >> (\hat{A}_1,d_{\hat{A}_1}) @>p >> (A,d_A) @>>>0\\
@. @| @V \zeta VV @| @.\\
0@>>> {(V,d_V)} @>i >> (\hat{A}_2,d_{\hat{A}_2}) @>p >> (A,d_A) @>>>0.
\end{CD}\]
\end{defn}

A {\bf section} of an abelian extension $(\hat{A},d_{\hat{A}})$ of $(A,d_A)$ by $(V,d_V)$ is a linear map $s:A\rar \hat{A}$ such that $p\circ s=\Id_A$.

Now for an abelian extension $(\hat{A},d_{\hat{A}})$ of $(A,d_A)$ by $(V,d_V)$ with a section $s:A\rar \hat{A}$, we define linear maps $\rho_l: A\to \mathrm{End}_\bk(V),~ x\mapsto (v\mapsto xv)$ and $\rho_r: A\to \mathrm{End}_\bk(V),~ x\mapsto (v\mapsto vx)$ respectively by
$$
xv:=s(x)v,\quad vx:=vs(x), \quad \forall x\in A, v\in V.
$$
\begin{prop}
  With the above notations, $(V,\rho_l,\rho_r,d_V)$ is a bimodule over the differential algebra $(A,d_A)$.
\end{prop}
\begin{proof}
For any $x,y\in A,\,v\in V$, since $s(xy)-s(x)s(y)\in V$ implies $s(xy)v=s(x)s(y)v$, we have
\[\rho_l(xy)(v)=s(xy)v=s(x)s(y)v=\rho_l(x)\circ\rho_l(y)(v).\]
Hence, $\rho_l$ is an algebra homomorphism. Similarly, $\rho_r$ is an algebra anti-homomorphism. Moreover, $d_{\hat{A}}(s(x))-s(d_A(x))\in V$ means that  $d_{\hat{A}}(s(x))v=s(d_A(x))v$. Thus we have
\begin{align*}
d_V(xv)&=d_V(s(x)v)=d_{\hat{A}}(s(x)v)=
d_{\hat{A}}(s(x))v+s(x)d_{\hat{A}}(v)+\lambda d_{\hat{A}}(s(x))d_{\hat{A}}(v)\\
&=s(d_A(x))v+s(x)d_V(v)+\lambda s(d_A(x))d_V(v)\\
&=d_A(x)v+xd_V(v)+\lambda d_A(x)d_V(v).
\end{align*}
Hence, $(V,\rho_l,\rho_r,d_V)$ is a bimodule over $(A,d_A)$.
\end{proof}

We  further  define linear maps $\psi:A\otimes A\rar V$ and $\chi:A\rar V$ respectively by
\begin{align*}
\psi(x,y)&=s(x)s(y)-s(xy),\quad\forall x,y\in A,\\
\chi(x)&=d_{\hat{A}}(s(x))-s(d_A(x)),\quad\forall x\in A.
\end{align*}
We transfer the differential algebra structure on $\hat{A}$ to $A\oplus V$ by endowing $A\oplus V$ with a multiplication $\cdot_\psi$ and a differential operator $d_\chi$ defined by
\begin{align}
\label{eq:mul}(x,u)\cdot_\psi(y,v)&=(xy,xv+uy+\psi(x,y)),\,\forall x,y\in A,\,u,v\in V,\\
\label{eq:dif}d_\chi(x,v)&=(d_A(x),\chi(x)+d_V(v)),\,\forall x\in A,\,v\in V.
\end{align}

\begin{prop}\label{prop:2-cocycle}
The triple $(A\oplus V,\cdot_\psi,d_\chi)$ is a differential algebra   if and only if
$(\psi,\chi)$ is a 2-cocycle  of the differential algebra $(A,d_A)$ with the coefficient  in $(V,d_V)$.
\end{prop}
\begin{proof}
If $(A\oplus V,\cdot_\psi,d_\chi)$ is a differential algebra, then the associativity of $\cdot_\psi$ implies
\begin{equation}
 \label{eq:mc}x\psi(y,z)-\psi(xy,z)+\psi(x,yz)-\psi(x,y)z=0.
\end{equation}Since $d_\chi$ satisfies \eqref{diff}, we deduce that
\begin{equation}
 \label{eq:dc}\chi(xy)-x\vdash_\lambda\chi(y)-\chi(x)\dashv_\lambda y +d_V(\psi(x,y))-\psi(d_A(x),y)-\psi(x,d_A(y))-\lambda \psi(d_A(x),d_A(y))=0.
\end{equation} Hence, $(\psi,\chi)$ is a  2-cocycle.

Conversely, if $(\psi,\chi)$ satisfies equalities~\eqref{eq:mc} and \eqref{eq:dc}, one can easily check that $(A\oplus V,\cdot_\psi,d_\chi)$ is a differential algebra.
\end{proof}

Now we are ready to classify abelian extensions of a differential algebra.

\begin{theorem}
Let $V$ be a vector space and  $d_V\in\End_\bk(V)$.
Then abelian extensions of a differential algebra $(A,d_A)$ by $(V,d_V)$ are classified by the second cohomology group ${\tilde{H}}_{\Diffl}^2(A,V)$ of $(A,d_A)$ with coefficients in the bimodule $(V,d_V)$.
\end{theorem}
\begin{proof}
Let $(\hat{A},d_{\hat{A}})$ be an abelian extension of $(A,d_A)$ by $(V,d_V)$. We choose a section $s:A\rar \hat{A}$ to obtain a 2-cocycle $(\psi,\chi)$ by Proposition~\ref{prop:2-cocycle}. We first show that the cohomological class of $(\psi,\chi)$ does not depend on the choice of sections. Indeed, let $s_1$ and $s_2$ be two distinct sections providing 2-cocycles $(\psi_1,\chi_1)$ and $(\psi_2,\chi_2)$ respectively. We define $\phi:A\rar V$ by $\phi(x)=s_1(x)-s_2(x)$. Then
\begin{align*}
\psi_1(x,y)&=s_1(x)s_1(y)-s_1(xy)\\
&=(s_2(x)+\phi(x))(s_2(y)+\phi(y))-(s_2(xy)+\phi(xy))\\
&=(s_2(x)s_2(y)-s_2(xy))+s_2(x)\phi(y)+\phi(x)s_2(y)-\phi(xy)\\
&=(s_2(x)s_2(y)-s_2(xy))+x\phi(y)+\phi(x)y-\phi(xy)\\
&=\psi_2(x,y)+\partial\phi(x,y)
\end{align*}
and
\begin{align*}
\chi_1(x)&=d_{\hat{A}}(s_1(x))-s_1(d_A(x))\\
&=d_{\hat{A}}(s_2(x)+\phi(x))-(s_2(d_A(x))+\phi(d_A(x)))\\
&=(d_{\hat{A}}(s_2(x))-s_2(d_A(x)))+d_{\hat{A}}(\phi(x))-\phi(d_A(x))\\
&=\chi_2(x)+d_V(\phi(x))-\phi(d_A(x))\\
&=\chi_2(x)-\delta\phi(x).
\end{align*}
That is, $(\psi_1,\chi_1)=(\psi_2,\chi_2)+\partial_{\Diffl}(\phi)$. Thus $(\psi_1,\chi_1)$ and $(\psi_2,\chi_2)$ are in the same cohomological class  {in $\tilde{H}_{\Diffl}^2(A,V)$}.

Next we prove that isomorphic abelian extensions give rise to the same element in  {$\tilde{H}_{\Diffl}^2(A,V)$.} Assume that $(\hat{A}_1,d_{\hat{A}_1})$ and $(\hat{A}_2,d_{\hat{A}_2})$ are two isomorphic abelian extensions of $(A,d_A)$ by $(V,d_V)$ with the associated homomorphism $\zeta:(\hat{A}_1,d_{\hat{A}_1})\rar (\hat{A}_2,d_{\hat{A}_2})$. Let $s_1$ be a section of $(\hat{A}_1,d_{\hat{A}_1})$. As $p_2\circ\zeta=p_1$, we have
\[p_2\circ(\zeta\circ s_1)=p_1\circ s_1=\Id_{A}.\]
Therefore, $\zeta\circ s_1$ is a section of $(\hat{A}_2,d_{\hat{A}_2})$. Denote $s_2:=\zeta\circ s_1$. Since $\zeta$ is a homomorphism of differential algebras such that $\zeta|_V=\Id_V$, we have
\begin{align*}
\psi_2(x,y)&=s_2(x)s_2(y)-s_2(xy)=\zeta(s_1(x))\zeta(s_1(y))-\zeta(s_1(xy))\\
&=\zeta(s_1(x)s_1(y)-s_1(xy))=\zeta(\psi_1(x,y))\\
&=\psi_1(x,y)
\end{align*}
and
\begin{align*}
\chi_2(x)&=d_{\hat{A}_2}(s_2(x))-s_2(d_A(x))=d_{\hat{A}_2}(\zeta(s_1(x)))-\zeta(s_1(d_A(x)))\\
&=\zeta(d_{\hat{A}_1}(s_1(x))-s_1(d_A(x)))=\zeta(\chi_1(x))\\
&=\chi_1(x).
\end{align*}
Consequently, all isomorphic abelian extensions give rise to the same element in {$\tilde{H}_{\Diffl}^2(A,V)$}.

Conversely, given two 2-cocycles $(\psi_1,\chi_1)$ and $(\psi_2,\chi_2)$, we can construct two abelian extensions $(A\oplus V,\cdot_{\psi_1},d_{\chi_1})$ and  $(A\oplus V,\cdot_{\psi_2},d_{\chi_2})$ via equalities~\eqref{eq:mul} and \eqref{eq:dif}. If they represent the same cohomological class {in $\tilde{H}_{\Diffl}^2(A,V)$}, then there exists a linear map $\phi:A\to V$ such that $$(\psi_1,\chi_1)=(\psi_2,\chi_2)+\partial_{\Diffl}(\phi).$$ Define $\zeta:A\oplus V\rar A\oplus V$ by
\[\zeta(x,v):=(x,\phi(x)+v).\]
Then $\zeta$ is an isomorphism of these two abelian extensions.
\end{proof}

\begin{remark}
In particular,   any vector space $V$ with linear endomorphism $d_V$  can serve as a trivial bimodule of $(A,d_A)$. In this situation,  central extensions  of $(A,d_A)$ by $(V,d_V)$  are classified by the second cohomology group $H_{\Diffl}^2(A,V)$ of $(A,d_A)$ with the coefficient in the trivial bimodule $(V,d_V)$. Note that for a trivial bimodule $(V,d_V)$, since $\partial_\lambda v=0$ for all $v\in V$, we have $$H_{\Diffl}^2(A,V)=\tilde{H}_{\Diffl}^2(A,V).$$
\end{remark}

\section{Deformations of differential algebras}\label{sec:def}

 In this section, we study formal deformations  of a differential algebra. In particular, we show that if the second cohomology group $\tilde{H}^2_{\Diffl}(A,A)=0$, then the differential algebra  $(A,d_A)$ is rigid.

 Let $(A,d_A)$ be a  differential algebra. Denote by $\mu_A$ the multiplication of $A$.
Consider the 1-parameterized family
$$\mu_t=\sum_{i=0}^{\infty} \mu_i t^i, \, \, \mu_i\in C^2_{\Diffl}(A, A),\quad
 d_t=\sum_{i=0}^{\infty} d_i t^i, \, \, d_i\in C^1_\diffo(d_A, d_A).$$

\begin{defn}
A {\bf 1-parameter formal deformation} of a differential algebra $(A, d_A)$ is a pair $(\mu_t, d_t)$ which endows the $\bk[[t]]$-module $(A[[t]], \mu_t, d_t)$ with the differential algebra structure  {over $\bk[[t]]$} such that $(\mu_0, d_0)=(\mu_A, d_A)$.
\end{defn}

 Given any differential algebra $(A,d_A)$, interpret
$\mu_A$ and $d_A$ as the formal power series
$\mu_t$ and $d_t$ with $\mu_i=\delta_{i,0}\mu_A$ and $d_i=\delta_{i,0}d_A$ respectively for all $i\geq0$. Then $(A[[t]],\mu_A,d_A)$ is a 1-parameter formal deformation of $(A, d_A)$.

 The pair $(\mu_t, d_t)$ generates a 1-parameter formal deformation  of the differential algebra $(A, d_A)$ if and only if for all $x, y, z\in A$, the following equalities hold:
\begin{eqnarray}\label{equation: ass nonexpanded}
 \mu_t(\mu_t(x, y), z)&=&\mu_t(x, \mu_t(y, z)),\\
\label{equation: derivation nonexpanded}
 d_t(\mu_t(x, y))&=&\mu_t(d_t(x), y)+\mu_t(x, d_t(y))+\lambda \mu_t(d_t(x),  d_t(y)).
\end{eqnarray}
Expanding these equations and collecting coefficients of $t^n$, we see that Eqs.~\eqref{equation: ass nonexpanded} and \eqref{equation: derivation nonexpanded} are equivalent to the systems of equations:
\begin{eqnarray}\label{equation: ass}
\sum_i^n \mu_i(\mu_{n-i}(x,y),z)&=&\sum_i^n\mu_i(x,\mu_{n-i}(y,z)),\\
\label{df}
\sum_{k,l\geq0\atop k+l=n}d_l\mu_k(x,y)&=&\sum_{k,l\geq0\atop k+l=n}\left(\mu_k(d_l(x),y)+\mu_k(x,d_l(y))\right)+\lambda\sum_{k,l,m\geq0\atop k+l+m=n}\mu_k(d_l(x),d_m(y)).
\end{eqnarray}

 \begin{remark}For $n=0$, Eq.~\eqref{equation: ass} is equivalent to the associativity of $\mu_A$, and Eq.~\eqref{df} is equivalent to the fact that $d_A$ is a $\lambda$-derivation.
\end{remark}

 \begin{prop}\label{prop:fddco}
Let $(A[[t]], \mu_t, d_t)$ be a $1$-parameter formal deformation of a differential algebra $(A,d_A)$. Then $(\mu_1, d_1)$ is a 2-cocycle of the differential algebra $(A,d_A)$ with the coefficient  in the regular bimodule $(A,d_A)$.	
\end{prop}
\begin{proof}
For $n =1$, Eq.~\eqref{equation: ass} is equivalent to $\partial \mu_1=0$, and Eq.~\eqref{df} is equivalent to $$\partial_\lambda d_1+\delta \mu_1=0.$$
Thus for $n=1$, Eqs.~\eqref{equation: ass} and \eqref{df} imply that $(\mu_1,d_1)$ is a 2-cocycle.	
\end{proof}

If $\mu_t=\mu_A$ in the above $1$-parameter formal deformation of the differential algebra $(A,d_A)$, we obtain a $1$-parameter formal deformation of the differential operator $d_A$. Consequently, we have

 \begin{coro}\label{coro:fdo}
Let $ d_t $ be a $1$-parameter formal deformation of the differential operator  $ d_A $.  Then $d_1$ is a 1-cocycle of the differential operator  $d_A$ with coefficients  in the regular bimodule $(A,d_A)$.	
\end{coro}
\begin{proof}
In the special case when $n =1$, Eq.~\eqref{df} is equivalent to $\partial_\lambda d_1=0$, which implies that  $d_1$ is a 1-cocycle of the differential operator  $d_A$ with coefficients in the regular bimodule $(A,d_A)$.	
\end{proof}

\begin{defn}
The $2$-cocycle $(\mu_1,d_1)$ is called the {\bf infinitesimal} of the $1$-parameter formal deformation $(A[[t]],\mu_t,d_t)$ of $(A,d_A)$.
\end{defn}

\begin{defn}
Let $(A[[t]],\mu_t,d_t)$ and $(A[[t]],\bar{\mu}_t,\bar{d}_t)$ be $1$-parameter formal deformations of $(A,d_A)$. A
{\bf formal isomorphism} from $(A[[t]],\bar{\mu}_t,\bar{d}_t)$ to $(A[[t]],\mu_t,d_t)$ is a power series $\Phi_t=\sum_{i\geq0}\phi_it^i:A[[t]]\lon A[[t]]$, where $\phi_i:A\to A$ are linear maps with $\phi_0=\Id_A$, such that
\begin{eqnarray}
\Phi_t\circ\bar{\mu}_t&=& \mu_t\circ(\Phi_t\times\Phi_t),\\
\Phi_t\circ\bar{d}_t&=&d_t\circ\Phi_t.
\end{eqnarray}

Two $1$-parameter formal deformations $(A[[t]],\mu_t,d_t)$ and $(A[[t]],\bar{\mu}_t,\bar{d}_t)$ are said to be {\bf equivalent} if  there exists a formal isomorphism $\Phi_t:(A[[t]],\bar{\mu}_t,\bar{d}_t) \lon (A[[t]],\mu_t,d_t)$.
\end{defn}

\begin{theorem}
The infinitesimals of two equivalent $1$-parameter formal deformations of $(A,d_A)$ are in the same cohomology class {in $\tilde{H}_{\Diffl}^2(A,A)$}.
\end{theorem}
\begin{proof}
Let $\Phi_t:(A[[t]],\bar{\mu}_t,\bar{d}_t)\lon(A[[t]],\mu_t,d_t)$ be a formal isomorphism. For all $x,y\in A$, we have
\begin{eqnarray*}
\Phi_t\circ\bar{\mu}_t(x,y)&=& \mu_t\circ(\Phi_t\times\Phi_t)(x,y),\\
\Phi_t\circ\bar{d}_t(x)&=& d_t\circ\Phi_t (x).
\end{eqnarray*}
 Expanding the above identities and comparing the coefficients of $t$, we obtain
\begin{eqnarray*}
\bar{\mu}_1(x,y)&=&\mu_1(x,y)+\phi_1(x)y+x\phi_1(y)-\phi_1(xy),\\
\bar{d}_1(x)&=&d_1(x)+d_A(\phi_1(x))-\phi_1(d_A(x)).
\end{eqnarray*}
Thus, we have $$(\bar{\mu}_1,\bar{d}_1)=(\mu_1,d_1)+\partial_{\Diffl}(\phi_1),$$ which implies that $[(\bar{\mu}_1,\bar{d}_1)]=[(\mu_1,d_1)]$ in $\tilde{H}^2_{\Diffl}(A,A)$.
\end{proof}

\begin{defn}
A $1$-parameter formal deformation $(A[[t]],\mu_t,d_t)$ of $(A,d_A)$ is said to be {\bf trivial} if it is equivalent to the  deformation $(A[[t]],\mu_A,d_A)$, that is, there exists  $\Phi_t=\sum_{i\geq0}\phi_it^i:A[[t]]\lon A[[t]]$, where  $\phi_i:A\to A$ are linear maps with $\phi_0=\Id_A$, such that
\begin{eqnarray}
\Phi_t\circ\mu_t&=&\mu_A\circ(\Phi_t\times\Phi_t),\\
\Phi_t \circ d_t&=&d_A\circ\Phi_t.
\end{eqnarray}
 \end{defn}

\begin{defn}
A differential algebra $(A,d_A)$ is said to be {\bf rigid} if every $1$-parameter formal deformation  is trivial.
\end{defn}

\begin{theorem}
Regarding $(A,d_A)$ as the regular bimodule over itself, if $\tilde{H}^2_{\Diffl}(A,A)=0$, the differential algebra  $(A,d_A)$ is rigid.
\end{theorem}
\begin{proof}
Let $(A[[t]],\mu_t,d_t)$ be a $1$-parameter formal deformation of $(A,d_A)$. By Proposition ~\ref{prop:fddco},   $(\mu_1,d_1)$ is a 2-cocycle. By $\tilde{H}^2_{\Diffl}(A,A)=0$, there exists a 1-cochain $\phi_1\in  C^1_\alg(A,A)$ such that
\begin{eqnarray}
\label{rigid}(\mu_1,d_1)=-\partial_{\Diffl}(\phi_1).
\end{eqnarray}
Then setting $\Phi_t=\Id_A+\phi_1 t$, we have a deformation $(A[[t]],\bar{\mu}_t,\bar{d}_t)$, where
\begin{eqnarray*}
\bar{\mu}_t(x,y)&=&\big(\Phi_t^{-1}\circ \mu_t\circ(\Phi_t\times\Phi_t)\big)(x,y),\\
\bar{d}_t(x)&=&\big(\Phi_t^{-1}\circ d_t\circ\Phi_t\big)(x).
\end{eqnarray*}
Thus, $(A[[t]],\bar{\mu}_t,\bar{d}_t)$ is equivalent to $(A[[t]],\mu_t,d_t)$. Moreover, we have
\begin{eqnarray*}
\bar{\mu}_t(x,y)&=&(\Id_A-\phi_1t+\phi_1^2t^{2}+\cdots+(-1)^i\phi_1^it^{i}+\cdots)(\mu_t(x+\phi_1(x)t,y+\phi_1(y)t)),\\
\bar{d}_t(x)&=&(\Id_A-\phi_1t+\phi_1^2t^{2}+\cdots+(-1)^i\phi_1^it^{i}+\cdots)(d_t(x+\phi_1(x)t)).
\end{eqnarray*}
Therefore,
\begin{eqnarray*}
\bar{\mu}_t(x,y)&=&xy+(\mu_1(x,y)+x\phi_1(y)+\phi_1(x)y-\phi_1(xy))t+\bar{\mu}_{2}(x,y)t^{2}+\cdots,\\
\bar{d}_t(x)&=&d_A(x)+(d_A(\phi_1(x))+d_1(x)-\phi_1(d_A(x)))t+\bar{d}_{2}(x)t^{2}+\cdots.
\end{eqnarray*}
By Eq.~\eqref{rigid}, we have
\begin{eqnarray*}
\bar{\mu}_t(x,y)&=&xy+\bar{\mu}_{2}(x,y)t^{2}+\cdots,\\
\bar{d}_t(x)&=&d_A(x)+\bar{d}_{2}(x)t^{2}+\cdots.
\end{eqnarray*}
Then by repeating the argument, we can show that $(A[[t]],\mu_t,d_t)$ is equivalent to $(A[[t]],\mu_A,d_A)$. Thus, $(A,d_A)$ is rigid.
\end{proof}

\section*{Appendix: Proof of Proposition~\mref{prop:delta}}

To simplify that notations, we use the abbreviation $x_{i,j}:=x_i,\dots, x_j,\,i\leq j,$
with the convention $x_{i,j}=1$ if $i>j$. For any $1\leq i_1<\cdots<i_k\leq n$ and $f\in C_\alg^n(A,V)$, define a function $f^{(i_1,\dots,i_k)}$ by
\begin{eqnarray*}
\lefteqn{f^{(i_1,\dots,i_k)}(x_1,\dots,x_n)}\\
&:=&f(x_1,\dots, x_{i_1-1}, d_A(x_{i_1}), x_{i_1+1}\dots, x_{i_2-1}, d_A(x_{i_2}), x_{i_2+1}\dots x_{i_k-1}, d_A(x_{i_k}), x_{i_k+1},\dots,x_n).
\end{eqnarray*}

In preparation for the proof of Proposition~\mref{prop:delta}, we first give two technical lemmas.
\begin{lemma}\label{lem:111}
For any $f\in C_\alg^n(A,V),\,x_1,\dots,x_{n+1}\in A$ with $n\geq1$, we have
\begin{align}
\nonumber&\sum_{k=1}^n\lambda^{k-1}\sum_{1\leq i_1<\cdots<i_k\leq n}\sum_{j=1}^n(-1)^jf^{(i_1,\dots,i_k)}(x_{1,j-1},x_jx_{j+1},x_{j+2,n+1}) \\
\nonumber =&\sum_{k=1}^{n}\lambda^{k-1}\sum_{1\leq i_1<\cdots<i_k\leq n+1}\sum_{1\leq j\leq n\atop j\neq i_r-1, i_r,\,\forall r}(-1)^jf(x_1, \dots, d_A(x_{i_1}), \dots, x_jx_{j+1},\dots, d_A(x_{i_k}), \dots, x_{n+1})\\
\nonumber  &+\sum_{k=1}^n\lambda^{k-1}\sum_{1\leq i_1<\cdots<i_k\leq n+1}\sum_{1\leq r\leq k\atop i_r-1\neq i_{r-1}}(-1)^{i_r-1}f(x_1,\dots,d_A(x_{i_1}),\dots,x_{i_r-1}d_A(x_{i_r}),\dots,d_A(x_{i_k}),\dots,x_{n+1})\\
\nonumber  &+\sum_{k=1}^n\lambda^{k-1}\sum_{1\leq i_1<\cdots<i_k\leq n+1}\sum_{1\leq r\leq k-1\atop i_r+1\neq i_{r+1}}(-1)^{i_r}f(x_1,\dots,d_A(x_{i_1}),\dots,d_A(x_{i_r})x_{i_r+1},\dots,d_A(x_{i_k}),\dots,x_{n+1})\\
\nonumber &+\sum_{k=2}^{n+1}\lambda^{k-1}\sum_{1\leq i_1<\cdots<i_k\leq n+1}\sum_{1\leq r\leq k-1\atop i_r+1= i_{r+1}}(-1)^{i_r}f(x_1,\dots,d_A(x_{i_1}),\dots,d_A(x_{i_r})d_A(x_{i_{r+1}}),\dots,d_A(x_{i_k}),\dots,x_{n+1})
\end{align}
\end{lemma}
\begin{proof}  In the second line of this proof, by convention $i_0=0, i_{k+1}=n+2$. By Eq.~\eqref{diff}, we have
\begin{align*}
&\sum_{k=1}^n\lambda^{k-1}\sum_{1\leq i_1<\cdots<i_k\leq n}\sum_{j=1}^n(-1)^jf^{(i_1,\dots,i_k)}(x_{1,j-1},x_jx_{j+1},x_{j+2,n+1})\\
=& \sum_{k=1}^n\hspace{-.1cm}
\lambda^{k-1}\hspace{-.65cm}
\sum_{1\leq i_1<\cdots<i_k\leq n}\hspace{-.2cm}
\sum_{1\leq j\leq n\atop {i_s<j<i_{s+1} \atop 0\leq s\leq k}}\hspace{-.25cm}
(-1)^j f( \dots, d_A(x_{i_1}), \dots, d_A(x_{i_s}), \dots, x_j x_{j+1}, \dots, d_A(x_{i_{s+1}+1}), \dots, d_A(x_{i_k+1}), \dots )\\
&+ \sum_{k=1}^n\lambda^{k-1}\sum_{1\leq i_1<\cdots<i_k\leq n} \sum_{r=1}^k (-1)^{i_r}f( \dots, d_A(x_{i_1}),\dots,  d_A(x_{i_r}x_{i_r+1}),\dots, d_A(x_{i_k+1}), \dots )\\
=& \sum_{k=1}^n\lambda^{k-1}\sum_{1\leq i_1<\cdots<i_k\leq n+1}\sum_{1\leq j\leq n\atop j\neq i_r, i_r-1,\,\forall r} (-1)^j f(  \cdots, d_A(x_{i_1}),    \dots, x_j x_{j+1}, \dots,    d_A(x_{i_k}), \dots )\\
&+ \sum_{k=1}^n\lambda^{k-1}\sum_{1\leq i_1<\cdots<i_k\leq n} \sum_{r=1}^k (-1)^{i_r}f(x_1,\dots, d_A(x_{i_1}),\dots,  d_A(x_{i_r})x_{i_r+1},\dots, d_A(x_{i_k+1}), \dots,  x_{n+1})\\
&+ \sum_{k=1}^n\lambda^{k-1}\sum_{1\leq i_1<\cdots<i_k\leq n} \sum_{r=1}^k (-1)^{i_r}f(x_1,\dots, d_A(x_{i_1}),\dots, x_{i_r}d_A(x_{i_r+1}),\dots, d_A(x_{i_k+1}), \dots,  x_{n+1})\\
&+ \sum_{k=1}^n\lambda^{k}\sum_{1\leq i_1<\cdots<i_k\leq n} \sum_{r=1}^k (-1)^{i_r}f(x_1,\dots, d_A(x_{i_1}),\dots,  d_A(x_{i_r})d_A(x_{i_r+1}),\dots, d_A(x_{i_k+1}), \dots,  x_{n+1})\\
=& \sum_{k=1}^n\lambda^{k-1}\sum_{1\leq i_1<\cdots<i_k\leq n+1}\sum_{1\leq j\leq n\atop j\neq i_r, i_r-1,\,\forall r} (-1)^j f(  \dots, d_A(x_{i_1}),    \dots, x_j x_{j+1}, \dots,    d_A(x_{i_k}), \dots )\\
  &+\sum_{k=1}^n\lambda^{k-1}\sum_{1\leq i_1<\cdots<i_k\leq n+1}\sum_{1\leq r\leq k\atop i_r-1\neq i_{r-1}}(-1)^{i_r-1}f(x_1,\dots,d_A(x_{i_1}),\dots,x_{i_r-1}d_A(x_{i_r}),\dots,d_A(x_{i_k}),\dots,x_{n+1})\\
  &+\sum_{k=1}^n\lambda^{k-1}\sum_{1\leq i_1<\cdots<i_k\leq n+1}\sum_{1\leq r\leq k-1\atop i_r+1\neq i_{r+1}}(-1)^{i_r}f(x_1,\dots,d_A(x_{i_1}),\dots,d_A(x_{i_r})x_{i_r+1},\dots,d_A(x_{i_k}),\dots,x_{n+1})\\
&+\sum_{k=1}^{n}\hspace{-.1cm}
\lambda^{k-1}\hspace{-.6cm}
\sum_{1\leq i_1<\cdots<i_{k+1}\leq n+1}\sum_{1\leq r\leq k\atop i_r+1= i_{r+1}}(-1)^{i_r}f(x_1,\dots,d_A(x_{i_1}),\dots,d_A(x_{i_r})d_A(x_{i_{r+1}}),\dots,d_A(x_{i_{k+1}}),\dots,x_{n+1})\\
=& \sum_{k=1}^n\lambda^{k-1}\sum_{1\leq i_1<\cdots<i_k\leq n+1}\sum_{1\leq j\leq n\atop j\neq i_r, i_r-1,\,\forall r} (-1)^j f(  \dots, d_A(x_{i_1}),    \dots, x_j x_{j+1}, \dots,    d_A(x_{i_k}), \dots )\\
  &+\sum_{k=1}^n\lambda^{k-1}\sum_{1\leq i_1<\cdots<i_k\leq n+1}\sum_{1\leq r\leq k\atop i_r-1\neq i_{r-1}}(-1)^{i_r-1}f(x_1,\dots,d_A(x_{i_1}),\dots,x_{i_r-1}d_A(x_{i_r}),\dots,d_A(x_{i_k}),\dots,x_{n+1})\\
  &+\sum_{k=1}^n\lambda^{k-1}\sum_{1\leq i_1<\cdots<i_k\leq n+1}\sum_{1\leq r\leq k-1\atop i_r+1\neq i_{r+1}}(-1)^{i_r}f(x_1,\dots,d_A(x_{i_1}),\dots,d_A(x_{i_r})x_{i_r+1},\dots,d_A(x_{i_k}),\dots,x_{n+1})\\
&+\sum_{k=2}^{n+1}\lambda^{k-1}\sum_{1\leq i_1<\cdots<i_k\leq n+1}\sum_{1\leq r\leq k-1\atop i_r+1= i_{r+1}}(-1)^{i_r}f(x_1,\dots,d_A(x_{i_1}),\dots,d_A(x_{i_r})d_A(x_{i_{r+1}}),\dots,d_A(x_{i_k}),\dots,x_{n+1}).
\qedhere\end{align*}
\end{proof}

\begin{lemma}
For any $f\in C_\alg^n(A,V),\,x_1,\dots,x_{n+1}\in A$ with $n\geq1$,
\begin{align}
&\sum_{k=1}^{n+1}\lambda^{k-1}\sum_{1\leq i_1<\cdots<i_k\leq n+1}(\partial f)^{(i_1, \dots, i_k)}(x_{1,{n+1}})-d_V (\partial f(x_{1,n+1}))\label{eq2}\\
\nonumber =&\sum_{k=1}^n\lambda^{k-1}\sum_{1\leq i_1<\cdots<i_k\leq n}x_1\vdash_\lambda f^{(i_1,\dots,i_k)}(x_{2,n+1})\\
\nonumber &+\sum_{k=1}^n\lambda^{k-1}\sum_{1\leq i_1<\cdots<i_k\leq n}\sum_{j=1}^n(-1)^jf^{(i_1,\dots,i_k)}(x_{1,j-1},x_jx_{j+1},x_{j+2,n+1})\\
\nonumber &+(-1)^{n+1}\sum_{k=1}^n\lambda^{k-1}\sum_{1\leq i_1<\cdots<i_k\leq n} f^{(i_1,\dots,i_k)}(x_{1,n})\dashv_\lambda x_{n+1}-x_1\vdash_\lambda d_V (f(x_{2,n+1}))\\
\nonumber &+\sum_{j=1}^n(-1)^{j-1}d_V(f(x_{1,j-1},x_jx_{j+1},x_{j+2,n+1}))
+(-1)^n d_V(f(x_{1,n}))\dashv_\lambda x_{n+1}.
\end{align}
\end{lemma}
 \begin{proof}
 As $d_V(xv)=d_A(x)v+x\vdash_\lambda d_V(v),\,d_V(vx)=vd_A(x)+d_V(v)\dashv_\lambda x$ for any $v\in V,x\in A$, we have
\begin{align*}
d_V (\partial f(x_{1,n+1}))&=d_A(x_1) f(x_{2,{n+1}})+x_1\vdash_\lambda d_V (f(x_{2,n+1}))\\
&+\sum_{j=1}^n(-1)^jd_V(f(x_{1,j-1},x_jx_{j+1},x_{j+2,n+1}))\\
&+(-1)^{n+1}f(x_{1,{n}})d_A(x_{n+1})+(-1)^{n+1} d_V(f(x_{1,n}))\dashv_\lambda x_{n+1}.
\end{align*}
Hence, we only need to check Eq.~\eqref{eq2} as follows. By Lemma \ref{lem:111}, we have
\begin{align*}
 &\sum_{k=1}^{n+1}\lambda^{k-1}\sum_{1\leq i_1<\cdots<i_k\leq n+1}(\partial f)^{(i_1, \dots, i_k)}(x_{1,{n+1}}) \\
 =& \lambda^n \partial f(d_A(x_1),\dots, d_A(x_{n+1}))
 +\sum_{k=1}^{n}\lambda^{k-1}\sum_{1\leq i_1<\cdots<i_k\leq n+1}(\partial f)^{(i_1, \dots, i_k)}(x_{1,{n+1}})\\
 =&\lambda^n d_A(x_1)  f(d_A(x_2),\dots, d_A(x_{n+1})) +\lambda^n \sum_{i=1}^n (-1)^i f(d_A(x_1), \dots, d_A(x_i)d_A(x_{i+1}),\dots, d_A(x_{n+1})) \\
 &+(-1)^{n+1} \lambda^n  f(d_A(x_1),\dots, d_A(x_{n}))d_A(x_{n+1}) \\
  &+d_A(x_1) f(x_{2,{n+1}})+\sum_{k=2}^{n}\lambda^{k-1}\sum_{2\leq i_2<\cdots<i_k\leq n+1}d_A(x_1) f^{(i_2-1, \dots, i_k-1)}(x_{2,{n+1}})\\
  &+\sum_{k=1}^{n}\lambda^{k-1}\sum_{2\leq i_1<\cdots<i_k\leq n+1}x_1 f^{(i_1-1, \dots, i_k-1)}(x_{2,{n+1}})\\
  &+\sum_{k=1}^{n}\lambda^{k-1}\sum_{1\leq i_1<\cdots<i_k\leq n+1}\sum_{1\leq j\leq n\atop j\neq i_r-1, i_r,\,\forall r}(-1)^jf(x_1, \dots, d_A(x_{i_1}), \dots, x_jx_{j+1},\dots, d_A(x_{i_k}), \dots, x_{n+1})\\
  &+\sum_{k=1}^n\lambda^{k-1}\sum_{1\leq i_1<\cdots<i_k\leq n+1}\sum_{1\leq r\leq k\atop i_r-1\neq i_{r-1}}(-1)^{i_r-1}f(x_1,\dots,d_A(x_{i_1}),\dots,x_{i_r-1}d_A(x_{i_r}),\dots,d_A(x_{i_k}),\dots,x_{n+1})\\
  &+\sum_{k=1}^n\lambda^{k-1}\sum_{1\leq i_1<\cdots<i_k\leq n+1}\sum_{1\leq r\leq k-1\atop i_r+1\neq i_{r+1}}(-1)^{i_r}f(x_1,\dots,d_A(x_{i_1}),\dots,d_A(x_{i_r})x_{i_r+1},\dots,d_A(x_{i_k}),\dots,x_{n+1})\\
&+\sum_{k=1}^{n}\lambda^{k-1}\sum_{1\leq i_1<\cdots<i_k\leq n+1}\sum_{1\leq r\leq k-1\atop i_r+1= i_{r+1}}(-1)^{i_r}f(x_1,\dots,d_A(x_{i_1}),\dots,d_A(x_{i_r})d_A(x_{i_{r+1}}),\dots,d_A(x_{i_k}),\dots,x_{n+1})\\
&+(-1)^{n+1}f(x_{1,{n}})d_A(x_{n+1}) +\sum_{k=2}^{n}\lambda^{k-1}(-1)^{n+1}\sum_{1\leq i_1<\cdots<i_{k-1}\leq n}  f^{(i_1, \dots, i_{k-1})}(x_{1,{n}})d_A(x_{n+1})\\
 &+\sum_{k=1}^{n}\lambda^{k-1}(-1)^{n+1}\sum_{1\leq i_1<\cdots<i_k\leq n} f^{(i_1, \dots, i_k)}(x_{1,{n}})x_{n+1}\\
 =& \lambda^n d_A(x_1)  f(d_A(x_2),\dots, d_A(x_{n+1})) +(-1)^{n+1} \lambda^n  f(d_A(x_1),\dots, d_A(x_{n}))d_A(x_{n+1})  \\
  &+d_A(x_1) f(x_{2,{n+1}})+\sum_{k=1}^{n-1}\lambda^{k}\sum_{1\leq i_1<\cdots<i_k\leq n}d_A(x_1) f^{(i_1, \dots, i_k)}(x_{2,{n+1}})\\
  &+\sum_{k=1}^{n}\lambda^{k-1}\sum_{1\leq i_1<\cdots<i_k\leq n}x_1 f^{(i_1, \dots, i_k)}(x_{2,{n+1}})\\
  &+ \sum_{k=1}^{n}\lambda^{k-1}\sum_{1\leq i_1<\cdots<i_k\leq n+1}\sum_{1\leq j\leq n\atop j\neq i_r-1, i_r,\,\forall r}(-1)^jf(x_1, \dots, d_A(x_{i_1}), \dots, x_jx_{j+1},\dots, d_A(x_{i_k}), \dots, x_{n+1})\\
  &+\sum_{k=1}^n\lambda^{k-1}\sum_{1\leq i_1<\cdots<i_k\leq n+1}\sum_{1\leq r\leq k\atop i_r-1\neq i_{r-1}}(-1)^{i_r-1}f(x_1,\dots,d_A(x_{i_1}),\dots,x_{i_r-1}d_A(x_{i_r}),\dots,d_A(x_{i_k}),\dots,x_{n+1})\\
  &+\sum_{k=1}^n\lambda^{k-1}\sum_{1\leq i_1<\cdots<i_k\leq n+1}\sum_{1\leq r\leq k-1\atop i_r+1\neq i_{r+1}}(-1)^{i_r}f(x_1,\dots,d_A(x_{i_1}),\dots,d_A(x_{i_r})x_{i_r+1},\dots,d_A(x_{i_k}),\dots,x_{n+1})\\
&+\sum_{k=1}^{n+1}\lambda^{k-1}\sum_{1\leq i_1<\cdots<i_k\leq n+1}\sum_{1\leq r\leq k-1\atop i_r+1= i_{r+1}}(-1)^{i_r}f(x_1,\dots,d_A(x_{i_1}),\dots,d_A(x_{i_r})d_A(x_{i_{r+1}}),\dots,d_A(x_{i_k}),\dots,x_{n+1})\\
&+(-1)^{n+1}f(x_{1,{n}})d_A(x_{n+1}) +\sum_{k=1}^{n-1}\lambda^{k}(-1)^{n+1}\sum_{1\leq i_1<\cdots<i_{k}\leq n}  f^{(i_1, \dots, i_{k})}(x_{1,{n}})d_A(x_{n+1})\\
 &+\sum_{k=1}^{n}\lambda^{k-1}(-1)^{n+1}\sum_{1\leq i_1<\cdots<i_k\leq n} f^{(i_1, \dots, i_k)}(x_{1,{n}})x_{n+1}\\
  =&d_A(x_1) f(x_{2,{n+1}})+ (-1)^{n+1}f(x_{1,{n}})d_A(x_{n+1})  \\
  &+\sum_{k=1}^{n}\lambda^{k-1}\sum_{1\leq i_1<\cdots<i_k\leq n}x_1 f^{(i_1, \dots, i_k)}(x_{2,{n+1}})\\
 & + \sum_{k=1}^{n}\lambda^{k}\sum_{1\leq i_1<\cdots<i_k\leq n}d_A(x_1) f^{(i_1, \dots, i_k)}(x_{2,{n+1}})\\
  &+\sum_{k=1}^{n}\lambda^{k-1}\sum_{1\leq i_1<\cdots<i_k\leq n+1}\sum_{1\leq j\leq n\atop j\neq i_r-1, i_r,\,\forall r}(-1)^jf(x_1, \dots, d_A(x_{i_1}), \dots, x_jx_{j+1},\dots, d_A(x_{i_k}), \dots, x_{n+1})\\
  &+\sum_{k=1}^n\lambda^{k-1}\sum_{1\leq i_1<\cdots<i_k\leq n+1}\sum_{1\leq r\leq k\atop i_r-1\neq i_{r-1}}(-1)^{i_r-1}f(x_1,\dots,d_A(x_{i_1}),\dots,x_{i_r-1}d_A(x_{i_r}),\dots,d_A(x_{i_k}),\dots,x_{n+1})\\
  &+\sum_{k=1}^n\lambda^{k-1}\sum_{1\leq i_1<\cdots<i_k\leq n+1}\sum_{1\leq r\leq k-1\atop i_r+1\neq i_{r+1}}(-1)^{i_r}f(x_1,\dots,d_A(x_{i_1}),\dots,d_A(x_{i_r})x_{i_r+1},\dots,d_A(x_{i_k}),\dots,x_{n+1})\\
&+\sum_{k=2}^{n+1}\lambda^{k-1}\sum_{1\leq i_1<\cdots<i_k\leq n+1}\sum_{1\leq r\leq k-1\atop i_r+1= i_{r+1}}(-1)^{i_r}f(x_1,\dots,d_A(x_{i_1}),\dots,d_A(x_{i_r})d_A(x_{i_{r+1}}),\dots,d_A(x_{i_k}),\dots,x_{n+1})\\
& +\sum_{k=1}^{n}\lambda^{k}(-1)^{n+1}\sum_{1\leq i_1<\cdots<i_{k}\leq n}  f^{(i_1, \dots, i_{k})}(x_{1,{n}})d_A(x_{n+1})\\
 &+\sum_{k=1}^{n}\lambda^{k-1}(-1)^{n+1}\sum_{1\leq i_1<\cdots<i_k\leq n} f^{(i_1, \dots, i_k)}(x_{1,{n}})x_{n+1}\\
=&d_A(x_1) f(x_{2,{n+1}})+ (-1)^{n+1}f(x_{1,{n}})d_A(x_{n+1})  \\
 &+\sum_{k=1}^n\lambda^{k-1}\sum_{1\leq i_1<\cdots<i_k\leq n}x_1  \vdash_\lambda f^{(i_1,\dots,i_k)}(x_{2,n+1})\\
&+\sum_{k=1}^n\lambda^{k-1}\sum_{1\leq i_1<\cdots<i_k\leq n}\sum_{j=1}^n(-1)^jf^{(i_1,\dots,i_k)}(x_{1,j-1},x_jx_{j+1},x_{j+2,n+1})\\
&+(-1)^{n+1}\sum_{k=1}^n\lambda^{k-1}\sum_{1\leq i_1<\cdots<i_k\leq n} f^{(i_1,\dots,i_k)}(x_{1,n})\dashv_\lambda x_{n+1}.
\qedhere\end{align*} \end{proof}

{\bf Proof of  Proposition~\mref{prop:delta}.}
For any $v\in C_\alg^0(A,V)=V$ and $x\in A$, we have
\begin{align*}
\delta(\partial v)(x)&=\partial v(d_A(x))-d_V(\partial v(x))\\
&=d_A(x)v-vd_A(x)-d_V(x v-v x)\\
&=d_A(x)v-vd_A(x)\\
&\ \ \ \ -d_A(x) v-xd_V(v)-\lambda d_A(x)d_V(v)+d_V(v)x+v d_A(x)+\lambda d_V(v)d_A(x)\\
&=-xd_V(v)-\lambda d_A(x)d_V(v)+d_V(v)x+ \lambda d_V(v)d_A(x)\\
&=-x\vdash_\lambda d_V(v)+d_V(v)\dashv_\lambda x\\
&=x\vdash_\lambda\delta v-\delta v\dashv_\lambda x\\
&=\partial_\lambda (\delta v)(x).
\end{align*}

For any $f\in C_\alg^n(A,V),\,x_1,\dots,x_{n+1}\in A$ with $n\geq1$,
\begin{align*}
\partial_\lambda(\delta f)(x_{1,n+1})=&x_1\vdash_\lambda\delta f(x_{2,n+1})+\sum_{j=1}^n(-1)^j\delta f(x_{1,j-1},x_jx_{j+1},x_{j+2,n+1})
+(-1)^{n+1}\delta f(x_{1,n})\dashv_\lambda x_{n+1}\\
=&\sum_{k=1}^n\lambda^{k-1}\sum_{1\leq i_1<\cdots<i_k\leq n}x_1\vdash_\lambda f^{(i_1,\dots,i_k)}(x_{2,n+1})\\
&+\sum_{j=1}^n\sum_{k=1}^n\lambda^{k-1}\sum_{1\leq i_1<\cdots<i_k\leq n}(-1)^jf^{(i_1,\dots,i_k)}(x_{1,j-1},x_jx_{j+1},x_{j+2,n+1})\\
&+(-1)^{n+1}\sum_{k=1}^n\lambda^{k-1}\sum_{1\leq i_1<\cdots<i_k\leq n} f^{(i_1,\dots,i_k)}(x_{1,n})\dashv_\lambda x_{n+1}
-x_1\vdash_\lambda d_V(f(x_{2,n+1}))\\
&+\sum_{j=1}^n(-1)^{j-1}d_V(f(x_{1,j-1},x_jx_{j+1},x_{j+2,n+1}))+(-1)^n d_V (f(x_{1,n}))\dashv_\lambda x_{n+1}.
\end{align*}
On the other hand, we have
\begin{align*}
\delta(\partial f)(x_{1,n+1})=&\sum_{k=1}^{n+1}\lambda^{k-1}\sum_{1\leq i_1<\cdots<i_k\leq n+1}(\partial f)^{(i_1, \dots, i_k)}(x_{1,{n+1}})-d_V (\partial f(x_{1,n+1})).
\end{align*}
Hence, by Eq.~\eqref{eq2},   $\partial_\lambda \delta=\delta\partial$.

\smallskip

\noindent
{{\bf Acknowledgments.} This work is supported in part by Natural Science Foundation of China (Grant Nos. 11501214, 11771142, 11771190,  11671139)   and STCSM (Grant Nos. 13dz2260400).  We give our warmest thanks to Rong Tang for useful comments.

\end{document}